\documentclass{amsart}[11pt]
\usepackage{graphicx,amscd, color,amsmath,amsfonts,amssymb,geometry,amssymb,graphicx}

\usepackage[initials]{amsrefs}

\providecommand{\U}[1]{\protect\rule{.1in}{.1in}}
\providecommand{\U}[1]{\protect\rule{.1in}{.1in}}
\newtheorem{theorem}{Theorem}[section]
\newtheorem{proposition}[theorem]{Proposition}
\newtheorem{corollary}[theorem]{Corollary}

\newtheorem{remark}[theorem]{Remark}

\newtheorem{lemma}[theorem]{Lemma}

\newtheorem{definition}[theorem]{Definition}
\numberwithin{equation}{section}

\DeclareMathOperator{\ext}{ext\,\!}
\begin{document}

\title[A geometric technique for the Bohnenblust--Hille inequalities]{A geometric technique to generate lower estimates for the constants in the Bohnenblust--Hille inequalities}

\author[G.A. Mu\~{n}oz-Fern\'{a}ndez, D. Pellegrino, J. Ramos Campos and J.B. Seoane-Sep\'{u}lveda]{G.A. Mu\~{n}oz-Fern\'{a}ndez, D. Pellegrino, J. Ramos Campos and J.B. Seoane-Sep\'{u}lveda}

\address{Departamento de An\'{a}lisis Matem\'{a}tico,\newline\indent Facultad de Ciencias Matem\'{a}ticas, \newline\indent Plaza de Ciencias 3, \newline\indent Universidad Complutense de Madrid,\newline\indent Madrid, 28040, Spain.}
\email{gustavo$\_$fernandez@mat.ucm.es}

\address{Departamento de Matem\'{a}tica, \newline\indent Universidade Federal da Para\'{\i}ba,
\newline\indent 58.051-900 - Jo\~{a}o Pessoa, Brazil.} \email{pellegrino@pq.cnpq.br and dmpellegrino@gmail.com}

\address{Departamento de Matem\'{a}tica, \newline\indent Universidade Federal da Para\'{\i}ba,
\newline\indent 58.051-900 - Jo\~{a}o Pessoa, Brazil.} \email{jamilsonrc@gmail.com}

\address{Departamento de An\'{a}lisis Matem\'{a}tico,\newline\indent Facultad de Ciencias Matem\'{a}ticas, \newline\indent Plaza de Ciencias 3, \newline\indent Universidad Complutense de Madrid,\newline\indent Madrid, 28040, Spain.}
\email{jseoane@mat.ucm.es}

\thanks{D. Pellegrino was supported by Supported by CNPq Grant 301237/2009-3, INCT-Matem\'{a}tica and CAPES-NF. G.A. Mu\~{n}oz-Fern\'{a}ndez and J. B. Seoane-Sep\'{u}lveda were supported by the Spanish Ministry of Science and Innovation, grant MTM2009-07848.}



\begin{abstract}
The Bohnenblust--Hille (polynomial and multilinear) inequalities were proved in 1931 in order to solve Bohr's absolute convergence problem on Dirichlet series. Since then these inequalities have found applications in various fields of analysis and analytic number theory. The control of the constants involved is crucial for applications, as it became evident in a recent outstanding paper of Defant, Frerick, Ortega-Cerd\'{a}, Ouna\"{\i}es and Seip published in 2011. The present work is devoted to obtain lower estimates for the constants appearing in the Bohnenblust--Hille polynomial inequality and some of its variants. The technique that we introduce for this task is a combination of the Krein--Milman Theorem with a description of the geometry of the unit ball of polynomial spaces on $\ell^2_\infty$.
\end{abstract}

\maketitle

\section{Preliminaries and background}

In 1913 H. Bohr proved that the maximal width $T$ of the vertical strip in
which a Dirichlet series $%
{\textstyle\sum\limits_{n=1}^{\infty}}
a_{n}n^{-s}$ converges uniformly but not absolutely is always less or equal
than $1/2$. Since then, the determination of the precise value of $T$ remained
a central problem in the study of Dirichlet series. Almost 20 years later, in
1931, H.F. Bohnenblust and E. Hille \cite{bh} showed that in fact $T=1/2.$ The technique used for this task was based on a puzzling generalization of Littlewood's
$4/3$ inequality to the framework of $m$-linear forms and homogeneous
polynomials.

The Bohnenblust--Hille inequality for homogeneous polynomials \cite{bh}
asserts that if $P:\ell_{\infty}^{N}\rightarrow\mathbb{C}$ is a $m$%
-homogeneous polynomial,
\[
P(z)={\textstyle\sum\limits_{\left\vert \alpha\right\vert =m}}a_{\alpha
}z^{\alpha},
\]
then there is a constant $D_{\mathbb{C},m}$ so that
\begin{equation}
\left(  {\textstyle\sum\limits_{\left\vert \alpha\right\vert =m}}\left\vert
a_{\alpha}\right\vert ^{\frac{2m}{m+1}}\right)  ^{\frac{m+1}{2m}}\leq
D_{\mathbb{C},m}\left\Vert P\right\Vert .\label{vvv}%
\end{equation}
The control of the estimates $D_{\mathbb{C},m},$ besides its challenging
nature, plays a decisive role in the theory: for instance, with adequate
estimates for $D_{\mathbb{C},m}$ in hands, Defant, Frerick, Ortega-Cerd\'{a},
Ouna\"{\i}es and Seip \cite{annals2011} were able to solve several important
questions related to Dirichlet series. In particular they obtained a
definitive generalization of a result of Boas and Khavinson \cite{boas},
showing that the $n$-dimension Bohr radius $K_{n}$ satisfies
\[
K_{n}\asymp\sqrt{\frac{\log n}{n}.}%
\]
The main result of \cite{annals2011} asserts that there is a $C>1$ such that
$D_{\mathbb{C},m}\leq C^{m}$ for all $m$, i.e., the Bohnenblust--Hille
inequality for homogeneous polynomials is hypercontractive. More precisely it
was shown that
\[
D_{\mathbb{C},m}\leq\left(  1+\frac{1}{m-1}\right)  ^{m-1}\sqrt{m}\left(
\sqrt{2}\right)  ^{m-1}%
\]
and, for example, one can take $C=2$ and it is simple to verify that
$D_{\mathbb{C},m}\leq2^{m}.$

It is worth mentioning that for small values of $m$, however, there are better estimates for $D_{\mathbb{C},m}$ due to Queff\'{e}lec \cite[Th. III-1]{Que}; for instance $D_{\mathbb{C},2}\leq1.7431$.


In view of the pivotal role played by the constants involved in the Bohnenblust--Hille inequality, a natural step forward is to try to obtain sharp constants and for this reason the search for lower estimates for the constants gains special importance. Moreover it is interesting to mention that, historically, the upper estimates obtained for the Bohnenblust--Hille inequalities have shown to be quite far from sharpness (see \cite{annals2011, psseo} for details).  Just to illustrate this fact, in the multilinear Bohnenblust--Hille inequality (complex case) the original upper estimate for the constant when $m=10$ is $80.28$ but now we know that this constant is not grater than $2.3$.

The multilinear version of Bohnenblust--Hille inequality is also an important subject of investigation in modern Functional Analysis and, as mentioned in \cite{defant}, \textquotedblleft \textit{it had and has deep applications in various fields of analysis, as for example in operator theory in Banach spaces, Fourier and harmonic analysis, complex analysis in finitely and infinitely many variables, and analytic number theory}\textquotedblright . For recent developments and related results we refer to \cite{DDGM,defant111,defant55,ff}.

Everything begins with Littlewood's famous $4/3$ theorem which asserts that for $\mathbb{K}=\mathbb{R}$ or
$\mathbb{C}$,
\[
\left(
{\displaystyle\sum\limits_{i,j=1}^{\infty}}
\left\vert A(e_{i},e_{j})\right\vert ^{\frac{4}{3}}\right)  ^{\frac{3}{4}}\leq
C_{\mathbb{K},2}\left\Vert A\right\Vert
\]
for every continuous bilinear form $A$ on $c_{0}\times c_{0}$, with
\[
C_{\mathbb{K},2}=\sqrt{2}.
\]
It is well-known that the power $4/3$ is optimal (see \cite{Ga}). For
real scalars it also can be shown that the constant $\sqrt{2}$ is optimal
(see \cite{Mu}). For complex scalars, however, there are several estimates for
$C_{\mathbb{C},2}$; below $K_{G}$ stands for the complex Grothendieck's
constant, and it is well-known that $1.338\leq K_{G}\leq1.405$ (see \cite{Di}):

\begin{itemize}
\item $C_{\mathbb{C},2}\leq\left(  K_{G}\sqrt{2}\right)  ^{1/2}$
(\cite[Theorem 34.11]{defant3} or \cite[Theorem 11.11]{TJ}),

\item $C_{\mathbb{C},2}\leq K_{G}$ (\cite[Corollary 2, p. 280]{LP}),

\item $C_{\mathbb{C},2}\leq\frac{2}{\sqrt{\pi}}\approx$ $1.128$
(\cite{defant2, Que}).
\end{itemize}

The optimal value for $C_{\mathbb{C},2}$ seems unknown. In 1931 Bohnenblust
and Hille \cite{bh} observed the connection between Littlewood's $4/3$ theorem and the so called Bohr's absolute convergence problem for Dirichlet series, which had been open for over 15 years. So, they generalized Littlewood's result to multilinear mappings, homogeneous polynomials and answered Bohr's problem.

Although the work of Bohnenblust and Hille is focused on complex scalars, it is
well-known that the result also holds for real scalars:

If $A$ is a continuous $n$-linear form on $c_{0}\times\cdots\times c_{0}$,
then there is a constant $C_{\mathbb{K},n}$ (depending only on $n$ and $\mathbb{K}$) such that%
\[
\left(
{\displaystyle\sum\limits_{i_{1},...,i_{m} = 1}^{\infty}}
\left\vert A(e_{i_{1}},...,e_{i_{n}})\right\vert ^{\frac{2n}{n+1}}\right)
^{\frac{n+1}{2n}}\leq C_{\mathbb{K},n}\left\Vert A\right\Vert .
\]

The estimates for $C_{\mathbb{K},n}$ were improved along the decades (see \cite{Davie, Ka, Que}). From recent
works (see \cite{Mu, psseo}) we know that, for real scalars,
\begin{align*}
C_{\mathbb{R},2}  &  =\sqrt{2}\approx1.414\\
1.587  &  \leq C_{\mathbb{R},3}\leq1.782\\
1.681  &  \leq C_{\mathbb{R},4}\leq2\\
1.741  &  \leq C_{\mathbb{R},5}\leq2.298\\
1.811  &  \leq C_{\mathbb{R},6}\leq2.520
\end{align*}
and, for the complex case,

\begin{align*}
C_{\mathbb{C},2}  &  \leq\left(  \frac{2}{\sqrt{\pi}}\right)  \approx1.128\\
C_{\mathbb{C},3}  &  \leq1.273\\
C_{\mathbb{C},4}  &  \leq1.437\\
C_{\mathbb{C},5}  &  \leq1.621\\
C_{\mathbb{C},10}  &  \leq2.292\\
C_{\mathbb{C},15}  &  \leq2.805.
\end{align*}
The lower bounds for $C_{\mathbb{R},m}$ obtained in \cite{Mu} are $2^{\frac{m-1}{m}%
}$, so the precise value for $C_{\mathbb{R},m}$ with ``big $m$'' is quite uncertain. Very recently, it was shown that for both real and complex scalars the
asymptotic behavior of the best values for $C_{\mathbb{K},n}$ is optimal \cite{DD}.

The (complex and real) Bohnenblust--Hille inequality can be re-written in the
context of multiple summing multilinear operators.

Let $X_{1},\ldots,X_{m}$ and $Y$ be Banach spaces over $\mathbb{K}=\mathbb{R}$
or $\mathbb{C}$, and $X^{\prime}$ be the topological dual of $X$. By
$\mathcal{L}(X_{1},\ldots,X_{m};Y)$ we denote the Banach space of all
continuous $m$-linear mappings from $X_{1}\times\cdots\times X_{m}$ to $Y$
with the usual sup norm. For $x_{1},...,x_{n}$ in $X$, let%

\[
\Vert(x_{j})_{j=1}^{n}\Vert_{w,1}:=\sup\{\Vert(\varphi(x_{j}))_{j=1}^{n}%
\Vert_{1}:\varphi\in X^{\prime},\Vert\varphi\Vert\leq1\}.
\]
If $1\leq p<\infty$, an $m$-linear mapping $U\in\mathcal{L}(X_{1},\ldots
,X_{m};Y)$ is multiple $(p;1)$-summing (denoted $\Pi_{(p;1)}(X_{1}%
,\ldots,X_{m};Y)$) if there exists a constant $U_{\mathbb{K},m}\geq0$ such
that
\begin{equation}
\left(  \sum_{j_{1},\ldots,j_{m}=1}^{N}\left\Vert U(x_{j_{1}}^{(1)}%
,\ldots,x_{j_{m}}^{(m)})\right\Vert ^{p}\right)  ^{\frac{1}{p}}\leq
U_{\mathbb{K},m}\prod_{k=1}^{m}\left\Vert (x_{j}^{(k)})_{j=1}^{N}\right\Vert
_{w,1} \label{lhs}%
\end{equation}
for every $N\in\mathbb{N}$ and any $x_{j_{k}}^{(k)}\in X_{k}$, $j_{k}%
=1,\ldots,N$, $k=1,\ldots,m$. The infimum of the constants satisfying
(\ref{lhs}) is denoted by $\left\Vert U\right\Vert _{\pi(p;1)}$. For $m=1$ we
recover the well-known concept of absolutely $(p;1)$-summing operators (see,
e.g. \cite{defant3, Di}).

The Bohnenblust--Hille inequality can be re-written in the context of multiple
summing multilinear operators in the following sense: every continuous
$m$-linear form $U:X_{1}\times\cdots\times X_{m}\rightarrow\mathbb{K}$ is
multiple $(\frac{2m}{m+1};1)$-summing. Moreover%
\begin{equation}
\left\Vert U\right\Vert _{\pi(\frac{2m}{m+1};1)}\leq C_{\mathbb{K}%
,m}\left\Vert U\right\Vert . \label{kko}%
\end{equation}
For details we refer to \cite{defant} and references therein.

From now on if $P:X\rightarrow Y$ is a $m$-homogeneous polynomial then $\overset{\vee}{P}$ denotes the (unique) symmetric $m$-linear map (also called the \textit{polar} of $P$) associated to $P$. Recall that an $m$-homogeneous polynomial $P:X\rightarrow Y$ is multiple
$(p;1)$-summing (denoted $\mathcal{P}_{(p;1)}(^{m}X;Y)$) if there exists a
constant $P_{\mathbb{K},m}\geq0$ such that%
\begin{equation}
\left(  \sum_{j_{1},\ldots,j_{m}=1}^{N}\left\Vert \overset{\vee}{P}(x_{j_{1}%
}^{(1)},\ldots,x_{j_{m}}^{(m)})\right\Vert ^{p}\right)  ^{\frac{1}{p}}\leq
P_{\mathbb{K},m}\prod_{k=1}^{m}\left\Vert (x_{j}^{(k)})_{j=1}^{N}\right\Vert
_{w,1} \label{lhs222}%
\end{equation}
for every $N\in\mathbb{N}$ and any $x_{j_{k}}^{(k)}\in X$, $j_{k}=1,\ldots,N$,
$k=1,\ldots,m$. The infimum of the constants satisfying (\ref{lhs222}) is
denoted by $\left\Vert P\right\Vert _{\pi(p;1)}$. Note that
\[
\left\Vert P\right\Vert _{\pi(p;1)}=\left\Vert \overset{\vee}{P}\right\Vert
_{\pi(p;1)}.%
\]

If $P\in\mathcal{P}(^{m}X;\mathbb{K})$ then $\overset{\vee}{P}\in
\mathcal{L}(^{m}X;\mathbb{K})=\Pi_{(\frac{2m}{m+1};1)}\left(  ^{m}%
X;\mathbb{K}\right)  $ and
\[
\left\Vert P\right\Vert _{\pi(\frac{2m}{m+1};1)}=\left\Vert \overset{\vee}%
{P}\right\Vert _{\pi(\frac{2m}{m+1};1)}\overset{\text{(\ref{kko})}}{\leq
}C_{\mathbb{K},m}\left\Vert \overset{\vee}{P}\right\Vert \leq\frac{m^{m}}%
{m!}C_{\mathbb{K},m}\left\Vert P\right\Vert .
\]
So, since $C_{\mathbb{K},m}$ does not depend on $X$ and $P$ we conclude that there
are constants $L_{\mathbb{K},m}$ (which does not depend on $X$ and $P$) such that%
\[
\left(  \sum_{j_{1},\ldots,j_{m}=1}^{N}\left\Vert \overset{\vee}{P}(x_{j_{1}%
}^{(1)},\ldots,x_{j_{m}}^{(m)})\right\Vert ^{p}\right)  ^{\frac{1}{p}}\leq
L_{\mathbb{K},m}\left\Vert P\right\Vert \prod_{k=1}^{m}\left\Vert (x_{j}%
^{(k)})_{j=1}^{N}\right\Vert _{w,1}.
\]

Note that if $X=\ell_{\infty}^{N},$ and $x^{(j)}=e_{j}$ for every $j=1,...,N$,
since
\[
\left\Vert (x^{(j)})_{j=1}^{N}\right\Vert _{w,1}=1,
\]
we have%
\begin{equation}
\left(  \sum_{j_{1},\ldots,j_{m}=1}^{N}\left\Vert \overset{\vee}{P}(e_{j_{1}%
},\ldots,e_{j_{m}})\right\Vert ^{\frac{2m}{m+1}}\right)  ^{\frac{m+1}{2m}}\leq
L_{\mathbb{K},m}\left\Vert P\right\Vert \label{fff}%
\end{equation}
for every $N\in\mathbb{N}$, which can be regarded as a kind of polynomial
Bohnenblust--Hille inequality.

Since (\ref{fff}) is confined to the
symmetric case, there is no obvious relation between the optimal values for
$C_{\mathbb{K},m}$ and the optimal values of $L_{\mathbb{K},m}.$

For $m=2$ it is well-known that $C_{\mathbb{R},2}=\sqrt{2}$. For $m>2$ the
precise values of $C_{\mathbb{R},m}$ are not known. Since
\[
L_{\mathbb{R},m}\leq\frac{m^{m}}{m!}C_{\mathbb{R},m},
\]
we have%
\begin{align*}
L_{\mathbb{R},2}  &  \leq2.828\\
L_{\mathbb{R},3}  &  \leq8.018\\
L_{\mathbb{R},4}  &  \leq21.333
\end{align*}

The main goal of this paper is to introduce a technique that helps to find nontrivial lower bounds for the constants involved in the Bohnenblust--Hille inequalities. Our approach is shown to be effective for the cases of $L_{\mathbb{R},m}$ and $D_{\mathbb{R},m}$. In the complex case we succeed in obtaining a lower bound for $D_{\mathbb{C},2}$.

More precisely, as a consequence of our estimates we show that if $D_{\mathbb{R},m}>0$ is such that
\[
\left(
{\textstyle\sum\limits_{\left\vert \alpha\right\vert =m}}
\left\vert a_{\alpha}\right\vert ^{\frac{2m}{m+1}}\right)  ^{\frac{m+1}{2m}%
}\leq D_{\mathbb{R},m}\left\Vert P\right\Vert ,
\]
for all $m$-homogeneous polynomial $P:\ell_{\infty}^{N}\rightarrow\mathbb{R}%
$,
\[
P(x)={\textstyle\sum\limits_{\left\vert \alpha\right\vert =m}}
a_{\alpha}x^{\alpha},
\]
then
\[
D_{\mathbb{R},m}\geq\left(  1.495\right)  ^{m}.\,
\]
Regarding to $L_{\mathbb{R},m},$ we show, for instance, that%
\begin{align*}
1.770  &  \leq L_{\mathbb{R},2}\\
1.453  &  \leq L_{\mathbb{R},3}\\
2.371  &  \leq L_{\mathbb{R},4}\\
3.272  &  \leq L_{\mathbb{R},8}\\
5.390  &  \leq L_{\mathbb{R},16}%
\end{align*}

In the complex case we show that
$D_{\mathbb{C},2}\geq1.1066$. So, combining this information with the best known upper estimate known for $D_{\mathbb{C},2}$ we conclude that $$ 1.1066\leq D_{\mathbb{C},2}\leq1.7431.$$

The techniques used in this paper in order to obtain good estimates for the constants $L_{\mathbb{K},n}$ and $D_{\mathbb{K},n}$ are based on the following result:

\begin{theorem}[consequence of Krein--Milman Theorem]\label{remark}
If $C$ is a convex body in a Banach space and $f:C\rightarrow
{\mathbb R}$ is a convex function that attains its maximum, then
there is an extreme point $e\in C$ so that $f(e)=\max\{f(x):x\in
C\}$.
\end{theorem}

This consequence of the Krein--Milman Theorem (\cite{KM}) provides good lower estimates on the constants $L_{\mathbb{K},n}$ when it is combined with a description of the geometry of the unit ball of a polynomial space on $\ell^m_\infty$. The problem of finding the extreme points of the unit ball of a polynomial space has been largely studied in the past few years. In particular, the following results will be particularly useful for our purpose.

\begin{theorem}[Choi \& Kim \cite{Choi}]\label{the:ExtPoints}
The extreme points of the unit ball of ${\mathcal P}(^2\ell_\infty^2)$ are the polynomials of the form
    $$
    \pm x^2,\ \pm y^2,\ \pm(tx^2-ty^2\pm2\sqrt{t(1-t)}xy),
    $$
with $t\in[1/2,1]$.
\end{theorem}

\begin{theorem}[G\'{a}mez-Merino, Mu\~{n}oz-Fern\'{a}ndez, S\'{a}nchez, Seoane-Sep\'{u}lveda \cite{Gamez}]\label{the:square}
If ${\mathcal P}(^2\Box)$ denotes the space ${\mathcal P}(^2{\mathbb R}^2)$ endowed with the sup norm over the unit interval $\Box=[0,1]^2$ and ${\mathsf B}_\Box$ is its unit ball, then the extreme points of ${\mathsf B}_\Box$ are
    $$
    \pm (tx^2-y^2+2\sqrt{1-t}xy)\quad \text{and}\quad \pm
    (-x^2+ty^2+2\sqrt{1-t}xy)\quad \text{with  $t\in[0,1]$}
    $$
or
    $$
    \pm(x^2+y^2-xy),\ \pm(x^2+y^2-3xy),\ \pm x^2,\ \pm y^2.
    $$
\end{theorem}

Note that Theorem \ref{the:square} is a kind on non-symmetric version of Theorem \ref{the:ExtPoints} and will be specially important when we are estimating the constants for $m\geq4$.
\section{Estimates for $L_{\mathbb{R},m}$}

In order to deal with polynomials and their polars we will introduce some
notation and a few basic results.

If $\alpha=(\alpha_{1},\ldots,\alpha_{n})\in{\mathbb{N}%
}^{*}$ then we define $|\alpha|:=\alpha_{1}+\cdots+\alpha_{n}$ and
    \[
    \binom{m}{\alpha}:=\frac{m!}{\alpha_{1}!\cdots\alpha_{n}!},
    \]
for $|\alpha|=m\in{\mathbb N}^*$.
Also, $\mathbf{x}^{\alpha}$ stands for the monomial $x_{1}^{\alpha_{1}}\cdots x_{n}^{\alpha_{n}}$
for $\mathbf{x}=(x_{1},\ldots,x_{n})\in{\mathbb{K}}^{n}$.
Having all this in mind, a straightforward consequence of the multinomial
formula yields the following relationship between the coefficients of a
homogeneous polynomial and the polar of the polynomial.

\begin{lemma}\label{lem:Multinomial}
If $P$ is a homogeneous polynomial of degree $n$ on ${\mathbb{K}}^{n}$ given
by
\[
P(x_{1},\ldots,x_{n})=\sum_{|\alpha|=m}a_{\alpha}\mathbf{{x}^{\alpha}, }%
\]
and $L$ is the polar of $P$, then
\[
L(e_{1}^{\alpha_{1}},\ldots,e_{n}^{\alpha_{n}})=\frac{a_{\alpha}}{\binom
{m}{\alpha}},
\]
where $\{e_{1},\ldots,e_{n}\}$ is the canonical basis of ${\mathbb{K}}^{n}$
and $e_{k}^{\alpha_{k}}$ stands for $e_{k}$ repeated $\alpha_{k}$ times.
\end{lemma}

\begin{definition}
Let us call $d$ the dimension of the space of all $m$-homogeneous polynomials
on ${\mathbb{R}}^{n}$. For every $m,n\in{\mathbb N}$, we define $\Phi_{m,n}:{\mathbb R}^d\rightarrow{\mathbb R}$ as follows:
Take ${\bf a}\in{\mathbb R}^d$ and consider the the $m$-homogeneous polynomial $P_{\mathbf{a}%
}(\mathbf{x})=\sum_{|\alpha|=m}a_{\alpha}\mathbf{x}^{\alpha}$ whose coefficients are the coordinates of ${\bf a}$. In order to avoid redundancies, assume that ${\bf a}=(a_\alpha)$ where the coordinates are arranged according to the lexicographic order of the $\alpha$'s. Then if $L_{\mathbf{a}}$ is the polar of $P_{\mathbf{a}}$ we define
    $$
    \Phi_{m,n}(\mathbf{a})  :=\left[  \sum_{i_{1}+\cdots+i_{m}=m}\left|
    L_{\mathbf{a}} (e_{i_{1}},\ldots,e_{i_{m}})\right|  ^{\frac{2m}{m+1}}\right]
    ^{\frac{m+1}{2m}}.
    $$
\end{definition}

\begin{remark}
Notice that Lemma \ref{lem:Multinomial} allows us to write $\Phi_{m,n}$ as
    \begin{align}
    \Phi_{m,n}(\mathbf{a}) & =\left[  \sum_{\overset{\alpha=(\alpha_{1},\cdots,\alpha_{n})}{|\alpha|=m}%
    }\binom{m}{\alpha}\left|  L_{\mathbf{a}} (e_{1}^{\alpha_{1}},\ldots
    ,e_{n}^{\alpha_{n}})\right|  ^{\frac{2m}{m+1}}\right]  ^{\frac{m+1}{2m}%
    }\nonumber\\
    & = \left[ \sum_{|\alpha|=m}\binom{m}{\alpha}\left| \frac{a_{\alpha}}%
    {\binom{m}{\alpha}}\right|  ^{\frac{2m}{m+1}}\right]  ^{\frac{m+1}{2m}%
    }.\label{ali:Phi}%
    \end{align}

Also $\Phi_{m,n}$ is, essentially, the composition of the norm in $\ell_{\frac
{2m}{m+1}}^{d}$ with the natural isomorphism between ${\mathcal{L}}^{s}%
(^{m}{\mathbb{R}}^{n})$ and ${\mathcal{P}}(^{m}{\mathbb{R}}^{n})$. Therefore
$\Phi_{m,n}$ is convex and by virtue of Krein--Milman Theorem
\[
L_{{\mathbb{R}},m}\geq L_{{\mathbb{R}},m}\left(  \ell_{\infty}^{n}\right)
:=\sup\{\Phi_{m,n}(\mathbf{a}):\mathbf{a}\in{\mathsf{B}}_{{\mathcal{P}}%
(^{m}\ell_{\infty}^{n})}\}=\sup\{\Phi_{m,n}(\mathbf{a}):\mathbf{a}%
\in\ext({\mathsf{B}}_{{\mathcal{P}}(^{m}\ell_{\infty}^{n})})\},
\]
where $\ext({\mathsf{B}}_{{\mathcal{P}}(^{m}\ell_{\infty}^{n})})$ is the set
of extreme points of ${\mathsf{B}}_{{\mathcal{P}}(^{m}\ell_{\infty}^{n})}$.
Observe that even in the case where the geometry of ${\mathsf{B}%
}_{{\mathcal{P}}(^{m}\ell_{\infty}^{n})}$ is not know, the mapping $\Phi_{m,n}$
provides a lower bound for $L_{{\mathbb{R}},m}$, namely
\begin{equation}
L_{{\mathbb{R}},m}\geq\frac{\Phi_{m,n}(\mathbf{a})}{\Vert P_{\mathbf{a}}\Vert
},\label{eq:lowe_bound}%
\end{equation}
for all $\mathbf{a}\in{\mathbb{R}}^{d}$.
\end{remark}

In
the following we will try to use the fact that the extreme points of
${\mathsf{B}}_{{\mathcal{P}}(^{m}\ell_{\infty}^{n})}$ have been characterized
for some choices of $m$ and $n$ (see for instance Theorem \ref{the:ExtPoints}%
).

\subsection{Case $m=2$}

\label{sub:Casem2}

We begin by illustrating that even sharp information for lower estimates for
$C_{\mathbb{R},2}$ may be useless for evaluating lower estimates for
$L_{\mathbb{R},2}.$ For instance, if $m=2$ in the multilinear
Bohnenblust--Hille inequality (in fact, Littlewood's $4/3$ inequality) the
best constant is $C_{\mathbb{R},2}=\sqrt{2}$ and this estimate is achieved
(see \cite{Mu}) when we use the bilinear form $T_{2}:\ell_{\infty}^{2}%
\times\ell_{\infty}^{2}\rightarrow\mathbb{R}$ given by
\[
T_{2}(x,y)=x_{1}y_{1}+x_{1}y_{2}+x_{2}y_{1}-x_{2}y_{2}.
\]
Note that $T_{2}$ is symmetric and the polynomial associated to $T_{2}$ is
$P_{2}:\ell_{\infty}^{2}\rightarrow\mathbb{R}$ given by%
\[
P_{2}(x)=x_{1}^{2}+2x_{1}x_{2}-x_{2}^{2}.
\]
Since $\left\Vert P_{2}\right\Vert =\left\Vert T_{2}\right\Vert =2,$ the
constant $L_{\mathbb{R},2}$ that appears for this choice of $P_{2}$ is again
$\sqrt{2}$, which is far from being a good lower estimate, as we shall see in
the next result, that gives the exact value for the constant $L_{{\mathbb R},2}(\ell_\infty^2)$.

\begin{theorem}\label{the:L_R,2}
    $
    L_{{\mathbb{R}},2}\geq    1.7700.
    $   More precisely,

    $$
    L_{{\mathbb{R}},2}\left(  \ell_{\infty}^{2}\right) =\sup\left\{  \left[  2t^{\frac{4}{3}}+2\left(  \sqrt{t(1-t)}\right)
    ^{\frac{4}{3}}\right]  ^{\frac{3}{4}}:t\in\lbrack1/2,1]\right\}\approx1.7700
    $$
and the supremum is attained at $t_{0}\approx0.9147$.
\end{theorem}

\begin{proof}
Observe that for polynomials in ${\mathcal{P}}(^{2}\ell_{\infty}^{2})$ of the
form $P_{\mathbf{a}}(x,y)=ax^{2}+by^{2}+cxy$ with $\mathbf{a}=(a,b,c)$ we
have
\begin{equation}
\Phi_{2,2}(a,b,c)=\left[  a^{\frac{4}{3}}+b^{\frac{4}{3}}+2\left(  \frac{c}%
{2}\right)  ^{\frac{4}{3}}\right]  ^{\frac{3}{4}}.\label{equ:Phi_22}%
\end{equation}
Using the Krein--Milman approach
    \[
    L_{{\mathbb{R}},2}\left(  \ell_{\infty}^{2}\right)  =\sup\{\Phi_{22}%
    (\mathbf{a}):\mathbf{a}\in\ext({\mathsf{B}}_{{\mathcal{P}}(^{2}\ell_{\infty
    }^{2})})\}.
    \]
Now, by Theorem \ref{the:ExtPoints},
$\ext({\mathsf{B}}_{{\mathcal{P}}(^{2}\ell_{\infty}^{2})})$ consists of the polynomials
    \[
    \pm(1,0,0),\quad\pm(0,1,0)\quad\text{and}\quad\pm(t,-t,\pm2\sqrt{t(1-t)}),
    \]
with $t\in\lbrack1/2,1]$. Since the contribution of $\pm(1,0,0)$ and
$\pm(0,1,0)$ to the supremum is irrelevant, we end up with
    \begin{align*}
    L_{{\mathbb{R}},2}\left(  \ell_{\infty}^{2}\right)   &  =\sup\{\Phi_{2,2}%
    (\pm(t,-t,\pm2\sqrt{t(1-t)})):t\in\lbrack1/2,1]\}\\
    &  =\sup\left\{  \left[  2t^{\frac{4}{3}}+2\left(  \sqrt{t(1-t)}\right)
    ^{\frac{4}{3}}\right]  ^{\frac{3}{4}}:t\in\lbrack1/2,1]\right\}.
    \end{align*}
The problem of maximizing explicitly this function is a hard one and the final
result is far from being good looking. The interested reader can obtain an
explicit solution in radical form using a variety of symbolic calculus packages, such as
Mathematica, Matlab or Maple. A 4-digit approximation yields
    \[
    L_{{\mathbb{R}},2}\geq L_{{\mathbb{R}},2}\left(  \ell_{\infty}^{2}\right)
    \approx1.7700,
    \]
where the maximum is attained at $t_{0}\approx0.9147$.
\end{proof}

\begin{remark}
A very good approximation of $L_{{\mathbb{R}},2}\left(  \ell_{\infty}^{2}\right)  $ can be obtained considering the polynomial
    \[
    P_{\mathbf{{a}}}(x,y)=x^{2}-y^{2}+xy,
    \]
i.e., $\mathbf{a}=(1,-1,1)$. It is easy to check that $\Vert P_{\mathbf{a}%
}\Vert=5/4$. Hence, using \eqref{eq:lowe_bound} we have
    \[
    L_{\mathbb{R},2}\left(  \ell_{\infty}^{2}\right)  \geq\frac{\Phi_{2,2}%
    (1,-1,1)}{\Vert P_{\mathbf{a}}\Vert}=\frac{4}{5}\cdot\left(  2+2\left(
    \frac{1}{2}\right)  ^{4/3}\right)  ^{3/4}\approx1.728.
    \]
\end{remark}

\subsection{Case $m=4$}\label{sub:Casem4}

In this section we calculate the exact value of $L_{{\mathbb R},4}$ in a subspace of ${\mathcal L}^s(^4\ell_\infty^2)$.
Observe that the value of $L_{{\mathbb R},4}$ in a subspace  is, obviously, a lower bound for
$L_{{\mathbb{R}},4}$.

\begin{theorem}
If $E=\{ax^{4}+by^{4}+cx^{2}y^{2}:a,b,c\in{\mathbb{R}}\}$  and $\overset{\vee}{E}$ is the space of polars of elements in $E$ endowed with the sup norm over the unit ball of $\ell_\infty^2$, then
    $$
    L_{{\mathbb R},4}(\overset{\vee}{E})=\left[  2+6\left(\frac{1}{2}\right)^{\frac{8}{5}%
    }\right]^{\frac{5}{8}}\approx2.371.
    $$
In particular $$L_{{\mathbb R},4}\geq L_{{\mathbb R},4}(\ell_\infty^2)\geq L_{{\mathbb R},4}(\overset{\vee}{E})\approx2.371.$$ Moreover, equality is attained in the Bohnenblust-Hille inequality in $\overset{\vee}{E}$
for the polars of the polynomials $P(x,y)=\pm
(x^{4}-y^{4}+3xy)$.
\end{theorem}

\begin{proof}
We just need to calculate
the maximum of $\Phi_{4,2}$ over $E$,
which is trivially isometric to the space ${\mathcal{P}}(^{2}\Box)$ (see
Theorem \ref{the:square} for the definition of ${\mathcal{P}}(^{2}\Box)$). If
$\Phi=\Phi_{4,2}|_{{\mathcal{P}}(^{2}\Box)}$, then $\Phi$ is obviously convex
and we have
\begin{align*}
L_{{\mathbb{R}},4}  &  \geq L_{{\mathbb{R}},4}\left(  \ell_{\infty}%
^{2}\right)  =\sup\{\Phi_{4,2}(\mathbf{a}):\mathbf{a}\in{\mathsf{B}%
}_{{\mathcal{P}}(^{4}\ell_{\infty}^{2})}\}\\
&  \geq\sup\{\Phi(\mathbf{a}):\mathbf{a}\in{\mathsf{B}}_{{\mathcal{P}}%
(^{2}\Box)}\}\\
&  =\sup\{\Phi(\mathbf{a}):\mathbf{a}\in\ext({\mathsf{B}}_{{\mathcal{P}}%
(^{2}\Box)})\},
\end{align*}
where the last equality is due to the Krein--Milman Theorem. Now by
\eqref{ali:Phi} we have
\[
\Phi(a,b,c)=\left[  a^{\frac{8}{5}}+b^{\frac{8}{5}}+6\left(  \frac{c}%
{6}\right)  ^{\frac{8}{5}}\right]  ^{\frac{5}{8}}.
\]
Using Theorem \ref{the:square} we obtain
\begin{align*}
\sup\{\Phi(\mathbf{a}):\mathbf{a}\in &  {\mathsf{B}}_{{\mathcal{P}}(^{2}\Box
)}\}\\
&  =\max\left\{  \left[  1+t^{\frac{8}{5}}+6\left(  \frac{\sqrt{1-t}}%
{3}\right)  ^{\frac{8}{5}}\right]  ^{\frac{5}{8}},\left[  2+6\left(  \frac
{1}{6}\right)  ^{\frac{8}{5}}\right]  ^{\frac{5}{8}},\left[  2+6\left(
\frac{1}{2}\right)  ^{\frac{8}{5}}\right]  ^{\frac{5}{8}}:t\in\lbrack
0,1]\right\} \\
&  =\left[  2+6\left(  \frac{1}{2}\right)  ^{\frac{8}{5}}\right]  ^{\frac
{5}{8}}.
\end{align*}
Observe that the maximum is attained at the polynomials $P(x,y)=\pm
(x^{4}-y^{4}+3xy)$. Hence we have proved that
\[
L_{{\mathbb{R}},4}\geq\left[  2+6\left(  \frac{1}{2}\right)  ^{\frac{8}{5}%
}\right]  ^{\frac{5}{8}}\approx2.371,
\]
moreover, a better (bigger) lower estimate for $L_{{\mathbb{R}},4}$ cannot be
obtained by considering polynomials of the form $ax^{4}+by^{4}+cx^{2}y^{2}$
with $a,b,c\in{\mathbb{R}}$.
\end{proof}


\subsection{Higher values of $m$}

The previous sections allow us to obtain lower estimates for $L_{{\mathbb R},m}$ for arbitrary large $m$'s. In this section we consider polynomials of the form $P_{2k}(x,y)=(ax^2+by^2+cxy)^k$. In the following, if $h\in{\mathbb Z}$,
$\left\lfloor h\right\rfloor $ denotes the biggest integer $H$ so that $H\leq
h$.

\begin{proposition}\label{pro:A_j}
If $P_{2k}(x,y)=(ax^2+by^2+cxy)^k$, then $P_{2k}(x,y)=\sum_{j=0}^{2k}A_jx^jy^{2k-j}$ with
    \begin{equation}\label{equ:Aj}
    A_{j}=\sum_{\ell=0}^{\left\lfloor \frac{j}{2}\right\rfloor }\frac
    {k!a^\ell b^{k-j+\ell}c^{j-2\ell}}{\ell!(j-2\ell)!(k-j+\ell)!},
    \end{equation}
for $j=0,\ldots,2k$.
\end{proposition}

\begin{proof}
Using the multinomial formula:
    $$
    P_{2k}(x,y)=(ax^2+by^2+cxy)^{k}=\sum_{\overset{\alpha_{1}+\alpha
    _{2}+\alpha_{3}=k}{\alpha_{1},\alpha_{2},\alpha_{3}\geq 0}}\frac{k!}{\alpha_{1}!\alpha_{2}!\alpha_{3}!}a^{\alpha_1}b^{\alpha_2}c^{\alpha_3}x^{2\alpha_{1}+\alpha_{3}}%
    y^{2\alpha_{2}+\alpha_{3}}.
    $$
Therefore, $x^{j}y^{2n-j}=x^{2\alpha_{1}+\alpha_{3}}y^{2\alpha_{2}%
+\alpha_{3}}$ for $j=1,\ldots,2k$ implies that
    $$
    \begin{cases}
    & 2\alpha_{1}+\alpha_{3}=j,\\
    & 2\alpha_{2}+\alpha_{3}=2k-j,
    \end{cases}
    $$
which, together with the fact that $\alpha_{1}+\alpha_{2}+\alpha_{3}=k$ and
$\alpha_{1},\alpha_{2},\alpha_{3}\geq0$ yield
    $$
    \begin{cases}
    & \alpha_{3}=j-2\alpha_{1},\\
    & \alpha_{2}=k-j+\alpha_{1},
    \end{cases}
    $$
with $\alpha_{1}=0,\ldots,\left\lfloor \frac{j}{2}\right\rfloor $. As a result of the previous comments, the coefficient $A_j$ is given by \eqref{equ:Aj}.
\end{proof}

\begin{corollary}
If $k\in{\mathbb N}$ then
    \begin{align}
    L_{{\mathbb R},2k}\geq     \left[\sum_{j=0}^{2k}\binom{2k}{j}\left|\frac{A_j}{\binom
    {2k}{j}}\right|^{\frac{4k}{2k+1}}\right]^{\frac
    {2k+1}{4k}},\label{ali:even}
    \end{align}
where
    $$
    A_{j}=\sum_{\ell=0}^{\left\lfloor \frac{j}{2}\right\rfloor }\frac
    {k!(-1)^{k-j+\ell}t_0^{k-j+2\ell}(2\sqrt{t_0(1-t_0)})^{j-2\ell}}{\ell!(j-2\ell)!(k-j+\ell)!},    $$
for $j=0,\ldots,2k$ and $t_0$ is as in Theorem \ref{the:L_R,2}.
\end{corollary}

\begin{proof}
If $P_{2k}(x,y)=(ax^2+by^2+cxy)^k$, using \eqref{ali:Phi}, \eqref{eq:lowe_bound} and Proposition \ref{pro:A_j} we arrive at
    \begin{align*}
    L_{{\mathbb R},2k}\geq \frac{1}{\|P_{2k}\|}    \left[\sum_{j=0}^{2k}\binom{2k}{j}\left|\frac{A_j}{\binom
    {2k}{j}}\right|^{\frac{4k}{2k+1}}\right]^{\frac
    {2k+1}{4k}},
    \end{align*}
with $A_j$ as in \eqref{equ:Aj}. Then the corollary follows by considering the polynomial $$P_{2k}(x,y)=(t_0x^2-t_0y^2+ 2\sqrt{t_0(1-t_0)}xy)^{2k},$$ which has norm 1.

\end{proof}

Hence \eqref{ali:even} provides a systematic formula to obtain a lower bound for $L_{{\mathbb R},m}$ for even $m$'s.

Observe that for $k=2$ we have
    $$
    L_{{\mathbb R},4}\geq\left[2t_0^\frac{16}{5}+6\left(\frac{2t_0-3t_0^2}{3}\right)^\frac{8}{5}+8(t_0\sqrt{t_0(1-t_0)})^\frac{8}{5}\right]^\frac{5}{8}\approx 2.1595,
    $$
which is a slightly worse constant than the one obtained in Section \ref{sub:Casem4}. Actually, the estimates \eqref{ali:even} can be improved for multiples of 4. Indeed, we just need to consider the polynomials
    $$
    Q_{4k}(x,y)=(ax^4+by^4+cx^2y^2)^k,
    $$
with $k\in{\mathbb N}$. Using exactly the same procedure described in this section
    \begin{align*}
    L_{{\mathbb R},4k}\geq \frac{1}{\|Q_{4k}\|}    \left[\sum_{j=0}^{2k}\binom{4k}{2j}\left|\frac{A_j}{\binom
    {4k}{2j}}\right|^{\frac{8k}{4k+1}}\right]^{\frac
    {4k+1}{8k}},
    \end{align*}
where the $A_j$'s, with $j=1,\ldots,2k$ are the same as in \eqref{equ:Aj}. Now, putting $a=1$, $b=1$ and $c=-3$, i.e., considering powers of the extreme polynomial that appeared in Section \ref{sub:Casem4}, we would have that $\|Q_{4k}\|=1$ for all $k\in{\mathbb N}$,  which proves the following:

\begin{theorem}
If $k\in{\mathbb N}$ then
    \begin{align}\label{ali:L_R4k}
    L_{{\mathbb R},4k}\geq    \left[\sum_{j=0}^{2k}\binom{4k}{2j}\left|\frac{B_j}{\binom
    {4k}{2j}}\right|^{\frac{8k}{4k+1}}\right]^{\frac
    {4k+1}{8k}},
    \end{align}
where
    \begin{equation}\label{equ:B_j}
    B_j=\sum_{\ell=0}^{\left\lfloor \frac{j}{2}\right\rfloor }\frac
    {k!(-3)^{j-2\ell}}{\ell!(j-2\ell)!(k-j+\ell)!},
    \end{equation}
for $j=0,\ldots,2k$.
\end{theorem}

As an example, let us apply \eqref{ali:L_R4k} and \eqref{equ:B_j} to obtain estimates for $L_{{\mathbb R},8}$ and $L_{{\mathbb R},12}$. The polynomials
are
\begin{align*}
Q_{8}(x,y)  &  =x^{8}-6x^{6}y^{2}+11x^{4}y^{4}-6x^{2}y^{6}+y^{8},\\
Q_{12}(x,y)  &  =x^{12}-9x^{10}y^{2}+30x^{8}y^{4}-45x^{6}y^{6}+30x^{4}%
y^{8}-9x^{2}y^{10}+y^{12}.
\end{align*}
Then
    \begin{align*}
    L_{{\mathbb{R}},8}&\geq\left[  2+2 \binom{8}{2}\left(  \frac{6}{\binom{8}{2}
    }\right)  ^{\frac{16}{9}}+\binom{8}{4}\left(  \frac{11}{\binom{8}{4}}\right)
    ^{\frac{16}{9}}\right]  ^{\frac{9}{16}}\approx3.2725.\\
    L_{{\mathbb{R}},12}  &  \geq\left[  2+2\binom{12}{2}\left(  \frac{9}%
    {\binom{12}{2}}\right)  ^{\frac{24}{13}}+2 \binom{12}{4}\left(  \frac
    {30}{\binom{12}{4}}\right)  ^{\frac{24}{13}} +\binom{12}{6}\left(  \frac
    {45}{\binom{12}{6}}\right)  ^{\frac{24}{13}}\right]  ^{\frac{13}{24}}
    \approx4.2441.
    \end{align*}
For higher degrees see Table \ref{tab:L_R,4k}.

\begin{table}[htb]
  \centering
\begin{tabular}
[c]{|l|l||l|l|}\hline
$k=4$ & $L_{\mathbb{R},16} \geq5.390975019$ & $k=40$ & $L_{\mathbb{R},160}
\geq16805.46318$\\\hline
$k=5$ & $L_{\mathbb{R},20} \geq6.787708182$ & $k=50$ & $L_{\mathbb{R},200}
\geq1.5654\times10^{5}$\\\hline
$k=6$ & $L_{\mathbb{R},24} \geq8.511696468$ & $k=60$ & $L_{\mathbb{R},240}
\geq1.4581\times10^{6}$\\\hline
$k=9$ & $L_{\mathbb{R},36} \geq16.65124974$ & $k=90$ & $L_{\mathbb{R},360}
\geq1.1781\times10^{9}$\\\hline
$k=10$ & $L_{\mathbb{R},40} \geq20.81051033$ & $k=100$ & $L_{\mathbb{R},400}
\geq1.0972\times10^{10}$\\\hline
\end{tabular}
\vspace{0.2cm}
\caption{$L_{\mathbb{R},4k}$ for some values of k}\label{tab:L_R,4k}
\end{table}

In order to clarify what the asymptotic growth of the sequence $\left(L_{\mathbb{R},4k}\right)_{k\in{\mathbb N}}$ is, a simple calculation of the quotients of the estimates obtained in Table \ref{tab:L_R,4k} for higher values of $k$ indicates that the ratio of the estimates on  $L_{\mathbb{R},4(k+1)}$ and $L_{\mathbb{R},4k}$  seem to tend to $\frac{5}{4}$.



\section{Estimates for $D_{\mathbb{R},m}$}

First observe that if ${\mathcal P}(^m\ell_\infty^n)$ has dimension $d$, then $D_{{\mathbb R},m}(\ell_\infty^n)$ is nothing but the optimal (smallest) equivalence constant between the spaces $\ell_\frac{2m}{m+1}^d$ and ${\mathcal P}(^m\ell_\infty^n)$. In other words, if we identify the polynomial $P_{\bf a}({\bf x})=\sum_{|\alpha|=m}a_{\alpha}\mathbf{x}^{\alpha}\in{\mathcal P}(^m\ell_\infty^n)$ with the vector ${\bf a}$ in ${\mathbb R}^d$ of all its coefficients, then
    \begin{equation}\label{equ:estimate_D_R,m}
    D_{{\mathbb R},m}(\ell_\infty^n)=\sup\left\{\frac{\|{\bf a}\|_\frac{2m}{m+1}}{\|P_{\bf a}\|}:P_{\bf a}\in {\mathcal P}(^m\ell_\infty^n)\right\}=
    \sup\left\{\|{\bf a}\|_\frac{2m}{m+1}:P_{\bf a}\in {\mathsf B}_{{\mathcal P}(^m\ell_\infty^n)}\right\},
    \end{equation}
where $\|\cdot\|_p$ denotes the $\ell_p$ norm. By convexity of $\|\cdot\|_p$ we also have
    \begin{equation}\label{equ:D_R,m}
    D_{{\mathbb R},m}(\ell_\infty^n)=
    \sup\left\{\|{\bf a}\|_\frac{2m}{m+1}:P_{\bf a}\in \ext({\mathsf B}_{{\mathcal P}(^m\ell_\infty^n)})\right\}.
    \end{equation}
As an easy consequence of Theorem \ref{the:ExtPoints} and \eqref{equ:D_R,m} we have:

\begin{theorem}
    $$
    D_{{\mathbb R},2}\geq D_{{\mathbb R},2}(\ell_\infty^2)=\sup\left\{\left[2t^\frac{4}{3}+\left(2\sqrt{t(1-t)}\right)^\frac{4}{3}\right]^\frac{3}{4}:t\in[1/2,1]\right\}\approx 1.8374.
    $$

\end{theorem}

The above supremum can be given explicitly in radical form using a symbolic calculus package, however the result is too lengthy to be shown. An excellent approximation can be obtain though in very simple terms considering the polynomial $P\in{\mathcal P}(^2\ell_\infty^2)$ defined by
    $$
    P(x,y)=x^2-y^2+xy.
    $$
Since $\|P\|=5/4$, from \eqref{equ:estimate_D_R,m} it follows that
    $$
    D_{\mathbb{R},2}\geq D_{\mathbb{R},2}(\ell_\infty^2)\geq \frac{\left(  3\right)  ^{3/4}}{5/4}\approx1.823.
    $$

\subsection{The case $m=3$}
Let us define $P_{3}:\ell_{\infty}^{6}\rightarrow\mathbb{R}$
by
    \[
    P_{3}(x)=(x_{1}+x_{2})\left(  x_{3}^{2}+x_{3}x_{4}-x_{4}^{2}\right)
    +(x_{1}-x_{2})\left(  x_{5}^{2}+x_{5}x_{6}-x_{6}^{2}\right)  .
    \]
We have $\Vert P_{3}\Vert=2\times\frac{5}{4}$. Also
    \[
    \left(
    {\textstyle\sum\limits_{\left\vert \alpha\right\vert =3}}
    \left\vert a_{\alpha}\right\vert ^{\frac{6}{4}}\right)  ^{\frac{4}{6}}\leq
    D_{\mathbb{R},3}\left\Vert P_{3}\right\Vert.
    \]
Therefore

\begin{proposition}
    \[
    D_{\mathbb{R},3}\geq\frac{\left(  4\times3\right)  ^{4/6}}{2\times\frac{5}{4}%
    }\approx2.096.
    \]
\end{proposition}

\subsection{The case $m=4$}

Acting as in Section \ref{sub:Casem4}, we can prove that the maximum value of $\frac{\|{\bf a}\|_\frac{8}{5}}{\|P_{\bf a}\|}$ where $P_{\bf a}$ ranges over the subspace of ${\mathcal P}(^4\ell_\infty^2)$ given by
    $$
    \{ax^4+by^4+cx^2y^2:a,b,c\in{\mathbb R}\},
    $$
is attained for the polynomial $Q_4(x,y)=x^4+y^4-3x^2y^2$. Hence, by \eqref{equ:estimate_D_R,m}, we have:

\begin{theorem}
If $E=\{ax^{4}+by^{4}+cx^{2}y^{2}:a,b,c\in{\mathbb{R}}\}$  is endowed with the sup norm over the unit ball of $\ell_\infty^2$, then
    $$
    D_{{\mathbb R},4}(E)=\|(1,1,-3)\|_\frac{8}{5}=\left(  2+\left(  3\right)  ^{8/5}\right)  ^{5/8}%
    \approx3.\,610.
    $$
In particular $$D_{{\mathbb R},4}\geq D_{{\mathbb R},4}(\ell_\infty^2)\geq D_{{\mathbb R},4}(E)\approx3.\,610.$$ Moreover, equality is attained in the polynomial Bohnenblust-Hille inequality in $E$
for the polynomials $P(x,y)=\pm (x^{4}-y^{4}+3xy)$.
\end{theorem}

\subsection{Higher values of $m$}

We consider again  the polynomials
    $$
    Q_{4k}(x,y)=\left(x^4+y^4 -3x^2y^2\right)^k,
    $$
for all $k\in{\mathbb N}$. Notice that $\| Q_{4k}\|=1$ for all $k\in{\mathbb N}$. Therefore, using \eqref{equ:estimate_D_R,m} together with the formula for the coefficients of the $Q_{4k}$ given by \eqref{equ:B_j}, we can obtain estimates for $D_{{\mathbb R},4k}$ with $k$ arbitrary (see Table \ref{tab3}). In fact we have:

\begin{theorem}
If $k\in{\mathbb N}$ then
    $$
    D_{{\mathbb R},4k}\geq    \left[\sum_{j=0}^{2k}\left|B_j\right|^{\frac{8k}{4k+1}}\right]^{\frac
    {4k+1}{8k}},
    $$
where
    $$
    B_j=\sum_{\ell=0}^{\left\lfloor \frac{j}{2}\right\rfloor }\frac
    {k!(-3)^{j-2\ell}}{\ell!(j-2\ell)!(k-j+\ell)!},
    $$
for $j=0,\ldots,2k$.
\end{theorem}

\vspace{0.2cm}
\begin{table}[!htb]
  \centering
\begin{tabular}{|l|l||l|l|}
  \hline
  $m=8$ & $D_{\mathbb{R},8} \geq 14.86998167$ & $m=80$ & $D_{\mathbb{R},80} \geq 3.0496\times10^{13}$ \\ \hline
  $m=12$ & $D_{\mathbb{R},12} \geq 66.39260961$ & $m=120$ & $D_{\mathbb{R},120} \geq 2.6821\times10^{20}$ \\ \hline
  $m=16$ & $D_{\mathbb{R},16} \geq 306.6665737$ & $m=160$ & $D_{\mathbb{R},160} \geq 2.4320\times10^{27}$ \\ \hline
  $m=20$ & $D_{\mathbb{R},20} \geq 1442.799763$ & $m=200$ & $D_{\mathbb{R},200} \geq 2.2443\times10^{34}$ \\ \hline
  $m=24$ & $D_{\mathbb{R},24} \geq 6866.770014$ & $m=240$ & $D_{\mathbb{R},240} \geq 2.0924\times10^{41}$ \\ \hline
  $m=28$ & $D_{\mathbb{R},28} \geq 32940.16505$ & $m=280$ & $D_{\mathbb{R},280} \geq 1.9649\times10^{48}$ \\ \hline
  $m=32$ & $D_{\mathbb{R},32} \geq 1.5892\times10^5$ & $m=320$ & $D_{\mathbb{R},320} \geq 1.8549\times10^{55}$ \\ \hline
  $m=36$ & $D_{\mathbb{R},36} \geq 7.7009\times10^5$ & $m=360$ & $D_{\mathbb{R},360} \geq 1.7582\times10^{62}$ \\ \hline
  $m=40$ & $D_{\mathbb{R},40} \geq 3.7444\times10^6$ & $m=400$ & $D_{\mathbb{R},400} \geq 1.6718\times10^{69}$ \\
  \hline
\end{tabular}
\vspace{0.2cm}
\caption{Estimates for $D_{\mathbb{R},m}$ for some values of $m$.}\label{tab3}
\end{table}

\vspace{0.2cm}
Obtaining more constants, we also get the following representation on the form $C^m$ of these lower bounds:

\vspace{0.2cm}
\begin{table}[!htb]
  \centering
\begin{tabular}[c]{|l|l||l|l|}
\hline
$m=8$ & $D_{\mathbb{R},8}\geq\left(  1.40132479\right)  ^{8}$ &  $m=5600$ & $D_{\mathbb{R},5600}\geq\left(  1.49475760 \right)  ^{5600}$ \\\hline

$m=200$ & $D_{\mathbb{R},200}\geq\left(  1.48509930\right)  ^{200}$ & $m=6400$ & $D_{\mathbb{R},6400}\geq\left(  1.49482368 \right)  ^{6400}$\\\hline

$m=800$ & $D_{\mathbb{R},800}\geq\left(  1,49212548\right)  ^{800}$ & $m=7200$ & $D_{\mathbb{R},7200}\geq\left(  1.49487590\right)  ^{7200}$\\\hline

$m=1600$ & $D_{\mathbb{R},1600}\geq\left(  1,49357368\right)  ^{1600}$ & $m=8000$ & $D_{\mathbb{R},8000}\geq\left(  1.49491825\right)  ^{8000}$ \\\hline

$m=3200$ & $D_{\mathbb{R},3200}\geq\left(  1.49437981 \right)  ^{3200}$ & $m=8800$ & $D_{\mathbb{R},8800}\geq\left(  1.49495333\right)  ^{8800}$ \\\hline

$m=4000$ & $D_{\mathbb{R},4000}\geq\left(  1.49455267 \right)  ^{4000}$ & $m=9600$ & $D_{\mathbb{R},9600}\geq\left(  1.49498289\right)  ^{9600}$\\\hline

$m=4800$ & $D_{\mathbb{R},4800}\geq\left(  1.49467111 \right)  ^{4800}$ & $m=12000$ & $D_{\mathbb{R},12000}\geq\left(  1.49504910 \right)  ^{12000}$ \\\hline
\end{tabular}
\vspace{0.2cm}
\caption{Estimates for $D_{\mathbb{R},m}$ in the form $D_{\mathbb{R},m} \geq C^m$.}\label{tab4}
\end{table}

\section{A lower estimate for $D_{\mathbb{C},2}$}

Let $P_{2}:\ell_{\infty}^{2}\left(  \mathbb{C}\right)  \rightarrow\mathbb{C}$
be a $2$-homogeneous polynomial given by
\[
P_{2}(z_{1},z_{2})=az_{1}^{2}+bz_{2}^{2}+cz_{1}z_{2}.
\]
with $a,b,c\in\mathbb{R}$. The following result can be obtained from a standard application of the
Maximum Modulus Principle together with \cite[eq. (3.1)]{AronKlimek}.

\begin{proposition}
If $P_{2}:\ell_{\infty}^{2}\left(  \mathbb{C}\right)  \rightarrow\mathbb{C}$
is defined by $P_{2}(z_{1},z_{2})=az_{1}^{2}+bz_{2}^{2}+cz_{1}z_{2}$ with
$a,b,c\in\mathbb{R}$, then
\[
\Vert P_{2}\Vert=
\begin{cases}
|a+b|+|c| & \text{if $ab\geq0$ or $|c(a+b)|>4|ab|$,}\\
\left(  |a|+|b|\right)  \sqrt{1+\frac{c^{2}}{4|ab|}} & \text{otherwise.}%
\end{cases}
\]
\end{proposition}

So, for these polynomials $P_{2}$ and $ab<0$ and $|c(a+b)|\leq4|ab|,$ the
Bohnenblust--Hille inequality is%
\[
\left(  \sqrt[3]{a^{4}}+\sqrt[3]{b^{4}}+\sqrt[3]{c^{4}}\right)  ^{\frac{3}{4}%
}\leq D_{\mathbb{C},2}\left(  \left\vert a\right\vert +\left\vert b\right\vert \right)
\sqrt{1+\frac{c^{2}}{4\left\vert ab\right\vert }}%
\]
and thus%
\[
D_{\mathbb{C},2}\geq\frac{\left(  \sqrt[3]{a^{4}}+\sqrt[3]{b^{4}}+\sqrt[3]{c^{4}}\right)
^{\frac{3}{4}}}{\left(  \left\vert a\right\vert +\left\vert b\right\vert
\right)  \sqrt{1+\frac{c^{2}}{4\left\vert ab\right\vert }}}.
\]
So, we must find real scalars $a,b,c$ so that $ab<0,$ $|c(a+b)|\leq4|ab|$ and%
\[
f_{2}(a,b,c)=\frac{\left(  \sqrt[3]{a^{4}}+\sqrt[3]{b^{4}}+\sqrt[3]{c^{4}%
}\right)  ^{\frac{3}{4}}}{\left(  \left\vert a\right\vert +\left\vert
b\right\vert \right)  \sqrt{1+\frac{c^{2}}{4\left\vert ab\right\vert }}}%
\]
is as big as possible. A straightforward examination shows that%
\[
f_{2}(a,b,c)<1.1067
\]
for all $a,b,c$ and, on the other hand,
\[
f_{2}(1,-1,\frac{352\,203}{125\,000})\approx1.1066.
\]
Combining the previous result and the known fact that $D_{\mathbb{C},2}\leq 1.7431$ we have the following result:
\begin{theorem}
$$1.1066\leq D_{\mathbb{C},2}\leq 1.7431.$$
\end{theorem}
\section{Final remarks}

In the real case we were able to deal with the case $m\geq2$ even in the absence of information on the geometry of the unit ball of ${\mathcal P}(^m\ell_\infty^n)$. However, in the complex case the technique seemed less effective for $m\geq3$. For obtaining lower estimates for $D_{\mathbb{C},m}$, with $m\geq3$ and sharper estimates for $D_{\mathbb{R},m}$, we believe that some effort should be made to get more information on the geometry of the unit ball of ${\mathcal P}(^m\ell_\infty^n)$ for higher values of $m,n$.

We do hope that the present work may serve as a motivation for future works investigating the geometry of the unit ball of complex and real polynomial spaces on $\ell^n_\infty$.



\begin{bibdiv}
\begin{biblist}

\bib{AronKlimek}{article}{
   author={Aron, Richard M.},
   author={Klimek, Maciej},
   title={Supremum norms for quadratic polynomials},
   journal={Arch. Math. (Basel)},
   volume={76},
   date={2001},
   number={1},
   pages={73--80},
}


\bib{boas}{article}{
   author={Boas, Harold P.},
   author={Khavinson, Dmitry},
   title={Bohr's power series theorem in several variables},
   journal={Proc. Amer. Math. Soc.},
   volume={125},
   date={1997},
   number={10},
   pages={2975--2979},
}

\bib{bh}{article}{
   author={Bohnenblust, H. F.},
   author={Hille, Einar},
   title={On the absolute convergence of Dirichlet series},
   journal={Ann. of Math. (2)},
   volume={32},
   date={1931},
   number={3},
   pages={600--622},
}

\bib{Choi}{article}{
   author={Choi, Yun Sung},
   author={Kim, Sung Guen},
   title={The unit ball of $\scr P(^2\!l^2_2)$},
   journal={Arch. Math. (Basel)},
   volume={71},
   date={1998},
   number={6},
   pages={472--480},
}

\bib{Davie}{article}{
   author={Davie, A. M.},
   title={Quotient algebras of uniform algebras},
   journal={J. London Math. Soc. (2)},
   volume={7},
   date={1973},
   pages={31--40},
}

\bib{DDGM}{article}{
   author={Defant, Andreas},
   author={D{\'{\i}}az, Juan Carlos},
   author={Garc{\'{\i}}a, Domingo},
   author={Maestre, Manuel},
   title={Unconditional basis and Gordon-Lewis constants for spaces of
   polynomials},
   journal={J. Funct. Anal.},
   volume={181},
   date={2001},
   number={1},
   pages={119--145},
}

\bib{defant3}{book}{
   author={Defant, Andreas},
   author={Floret, Klaus},
   title={Tensor norms and operator ideals},
   series={North-Holland Mathematics Studies},
   volume={176},
   publisher={North-Holland Publishing Co.},
   place={Amsterdam},
   date={1993},
   pages={xii+566},
}


\bib{annals2011}{article}{
   author={Defant, Andreas},
   author={Frerick, Leonhard},
   author={Ortega-Cerd\`{a}, Joaquim},
   author={Ouna{\"{\i}}es, Myriam},
   author={Seip, Kristian},
   title={The Bohnenblust-Hille inequality for homogeneous polynomials is
   hypercontractive},
   journal={Ann. of Math. (2)},
   volume={174},
   date={2011},
   number={1},
   pages={485--497},
}


\bib{defant111}{article}{
   author={Defant, Andreas},
   author={Garc{\'{\i}}a, Domingo},
   author={Maestre, Manuel},
   author={P\'{e}rez-Garc{\'{\i}}a, David},
   title={Bohr's strip for vector valued Dirichlet series},
   journal={Math. Ann.},
   volume={342},
   date={2008},
   number={3},
   pages={533--555},
}


\bib{ff}{article}{
   author={Defant, Andreas},
   author={Garc{\'{\i}}a, Domingo},
   author={Maestre, Manuel},
   author={Sevilla-Peris, Pablo},
   title={Bohr's strips for Dirichlet series in Banach spaces},
   journal={Funct. Approx. Comment. Math.},
   volume={44},
   date={2011},
   number={part 2},
   part={part 2},
   pages={165--189},
}


\bib{defant55}{article}{
   author={Defant, Andreas},
   author={Maestre, Manuel},
   author={Schwarting, U.},
   title={Bohr radii for vector valued holomorphic functions},
   status={Preprint},
}


\bib{defant}{article}{
   author={Defant, Andreas},
   author={Popa, Dumitru},
   author={Schwarting, Ursula},
   title={Coordinatewise multiple summing operators in Banach spaces},
   journal={J. Funct. Anal.},
   volume={259},
   date={2010},
   number={1},
   pages={220--242},
}



\bib{defant2}{article}{
   author={Defant, Andreas},
   author={Sevilla-Peris, Pablo},
   title={A new multilinear insight on Littlewood's 4/3-inequality},
   journal={J. Funct. Anal.},
   volume={256},
   date={2009},
   number={5},
   pages={1642--1664},
}



\bib{Di}{book}{
   author={Diestel, Joe},
   author={Jarchow, Hans},
   author={Tonge, Andrew},
   title={Absolutely summing operators},
   series={Cambridge Studies in Advanced Mathematics},
   volume={43},
   publisher={Cambridge University Press},
   place={Cambridge},
   date={1995},
   pages={xvi+474},
}

\bib{DD}{article}{
author={Diniz, D.},
author={Mu\~{n}oz-Fern\'{a}ndez, G. A.},
author={Pellegrino, D.},
author={Seoane-Sep\'{u}lveda, J. B.},
title={The asymptotic growth of the constants in the Bohnenblust-Hille inequality is optimal},
journal={arXiv:1108.1550v2 [math.FA]},
}

\bib{Mu}{article}{
author={Diniz, D.},
author={Mu\~{n}oz-Fern\'{a}ndez, G. A.},
author={Pellegrino, D.},
author={Seoane-Sep\'{u}lveda, J. B.},
title={Lower bounds for the constants in the Bohnenblust--Hille
inequality: the case of real scalars},
status={arXiv:1111.3253v2  [math.FA]},
}

\bib{Gamez}{article}{
author={G\'{a}mez-Merino, J. L.},
author={Mu\~{n}oz-Fern\'{a}ndez, G. A.},
author={S\'{a}nchez, V.},
author={Seoane-Sep\'{u}lveda, J. B.},
title={Inequalities for polynomials on the unit square via the Krein--Milman Theorem},
status={Preprint},
}

\bib{Ga}{book}{
   author={Garling, D. J. H.},
   title={Inequalities: a journey into linear analysis},
   publisher={Cambridge University Press},
   place={Cambridge},
   date={2007},
   pages={x+335},
}

\bib{Ka}{article}{
   author={Kaijser, Sten},
   title={Some results in the metric theory of tensor products},
   journal={Studia Math.},
   volume={63},
   date={1978},
   number={2},
   pages={157--170},
   issn={0039-3223},
   review={\MR{511301 (80c:46069)}},
}

\bib{KM}{article}{
   author={Krein, M.},
   author={Milman, D.},
   title={On extreme points of regular convex sets},
   journal={Studia Math.},
   volume={9},
   date={1940},
   pages={133--138},
}



\bib{LP}{article}{
   author={Lindenstrauss, J.},
   author={Pe{\l}czy{\'n}ski, A.},
   title={Absolutely summing operators in $L_{p}$-spaces and their
   applications},
   journal={Studia Math.},
   volume={29},
   date={1968},
   pages={275--326},
}

\bib{Litt}{article}{
   author={Littlewood, J.E.},
   title={On bounded bilinear forms in an infinite number of variables},
   journal={Q. J. Math.},
   volume={1},
   date={1930},
   pages={164--174},
}

\bib{lama11}{article}{
author={Mu\~{n}oz-Fern\'{a}ndez, G. A.},
author={Pellegrino, D.},
author={Seoane-Sep\'{u}lveda, J. B.},
title={Estimates for the asymptotic behavior of the constants in the
Bohnenblust--Hille inequality},
journal={Linear Multilinear Algebra},
status={In Press}
}


\bib{trin}{article}{
   author={Mu\~{n}oz-Fern\'{a}ndez, Gustavo A.},
   author={Seoane-Sep\'{u}lveda, Juan B.},
   title={Geometry of Banach spaces of trinomials},
   journal={J. Math. Anal. Appl.},
   volume={340},
   date={2008},
   number={2},
   pages={1069--1087},
   issn={0022-247X},
   review={\MR{2390911 (2008m:46014)}},
   doi={10.1016/j.jmaa.2007.09.010},
}


\bib{psseo}{article}{
   author={Pellegrino, Daniel},
   author={Seoane-Sep\'{u}lveda, Juan B.},
   title={New upper bounds for the constants in the Bohnenblust-Hille
   inequality},
   journal={J. Math. Anal. Appl.},
   volume={386},
   date={2012},
   number={1},
   pages={300--307},
}

\bib{Que}{article}{
author={Queff\'{e}lec, H.},
title={H. Bohr's vision of ordinary Dirichlet series: old and new results},
journal={J. Anal.},
volume={3},
date={1995},
pages={43--60},
}

\bib{TJ}{book}{
   author={Tomczak-Jaegermann, Nicole},
   title={Banach-Mazur distances and finite-dimensional operator ideals},
   series={Pitman Monographs and Surveys in Pure and Applied Mathematics},
   volume={38},
   publisher={Longman Scientific \& Technical},
   place={Harlow},
   date={1989},
   pages={xii+395},
}

\end{biblist}
\end{bibdiv}
\end{document}